\documentclass[12pt,pdftex]{article} 

\usepackage{0header}

\title{\textbf{Boundary regularity for anisotropic minimal Lipschitz graphs}}

\author[1]{Antonio De Rosa}\affil[1]{\small\textit {Department of Mathematics, University of Maryland, College Park, MD 20742, USA} }
\author[2]{Reinaldo Resende}\affil[2]{\small\textit {Department of Mathematics, Carnegie Mellon University, Pittsburgh, PA 15213, USA}}

\begin{document}

\date{}
\maketitle

\begin{abstract}
We prove that $m$-dimensional Lipschitz graphs in any codimension with $C^{1,\alpha}$ boundary and anisotropic mean curvature bounded in $L^p$, $p > m$, are regular at every boundary point with density bounded above by $1/2 +\sigma$, provided the anisotropic energy satisfies the uniform scalar atomic condition.
\end{abstract}


\section{Introduction}

\subsection{Regularity theorems for the area functional}\label{Sec1}
In his seminal work \cite{All}, Allard developed the regularity theory for varifolds with bounded first variation. He first obtained a rectifiability theorem, proving that, for every $m$-varifold $\bV$,
\begin{equation}\label{R}\eqname{(R)}
\mbox{if} \, \sup_{\|X\|_{\infty}\leq 1}\delta\bV(X) \leq 1 \ , \mbox{then} \, \bV\res\{x\in\R^{m+n}: \udens{m}{\bV}{x}>0\} \, \mbox{is a rectifiable varifold}.
\end{equation}

Additionally, he proved a celebrated $\varepsilon$-regularity theorem, which guarantees, for every $m$-varifold $\bV$ with generalized mean curvature in $\Lp{p}(\cH^m)$, $p>m$, and $\cH^m(\spt(\|\bV\|)\cap \ball{x}{r})$ close to $\omega_m r^m$, that $\spt(\|\bV\|)$ is $C^{1,\eta}$ locally around $x$ for some $\eta\in (0,1)$.  

Afterwards, in \cite{AllB}, Allard extended this regularity result to varifolds with $C^{1,1}$ boundary. Here the boundary is intended as a $C^{1,1}$ submanifold $\Gamma$ with dimension $m-1$ such that the first variation of the varifold is bounded away from $\Gamma$. 

One of the reasons why Allard considered a $C^{1,1}$ boundary is that for each point $x\in \Gamma$ there is a neighborhood of $x$ in $\Gamma$ such that the distance function $y\mapsto  \operatorname{dist}(y,\Gamma)$ is differentiable in a tubular neighborhood of $\Gamma$. For more details, we refer the reader to \cite{ghomi2022total}, where the authors explore Federer's notion of reach of $\Gamma$ to prove that $\Gamma$ is $C^{1,1}$ if, and only if, the reach is strictly positive. Bourni \cite{Theodora} generalized Allard's boundary regularity theorem to  $C^{1,\alpha}$ boundaries, for $\alpha\in (0,1)$, using a Whitney partition argument to overcome the non-differentiability of the distance function above around $\Gamma$.

\subsection{Anisotropic functionals}\label{Sec2}

A natural question is whether or not the regularity theorems mentioned in Section \ref{Sec1} still hold if the first variation is not computed with respect to the area functional, but rather with respect to more general anisotropic functionals $\cF:\R^{m+n}\times\Gr(m,n)\to (0,+\infty)$. 

Anisotropic functionals, together with their minimizers and critical points, have been extensively studied, and several results available for the area functional have been extended to the anisotropic setting.
This is typically not an easy task, as several basic properties of isotropic minimal surfaces dramatically fail for anisotropic minimal surfaces. More precisely, Allard's proof of the aforementioned regularity theorems strongly rely on the well-known \emph{monotonicity formula}. However, in \cite{allard1974characterization}, Allard showed that the monotonicity formula holds only for linear transformations of the area functional. The lack of a monotonicity formula for general anisotropic functionals gives rise to numerous technical issues in the theory, since the majority of the isotropic results \emph{deeply} rely on it. 

De Philippis, De Rosa and Ghiraldin proved in \cite{de2018rectifiability} that, if $\cF$ is of class $C^{1}$ and satisfies the so called \emph{atomic condition (AC)}, the rectifiability criterium \ref{R} holds also for the anisotropic first variation $\delta_\cF$ in place of $\delta$. This result found applications, among others, in the solution of the anisotropic Plateau problem \cite{de2020existence,DePDeR} and the anisotropic min-max theory \cite{DeR}. In the case of an autonomous anisotropy $\cF$, i.e., $\cF$ does not depend on the variable in $\R^{m+n}$, the authors in \cite{de2018rectifiability} showed that the validity of \ref{R} is actually equivalent to AC. 
We refer the interested reader to the following works for further developments of the theory: \cite{harrison2017general,de2006lecture,de2020existence,DRL,tione,de2019direct}.
In codimension $n=1$ and in dimension $m=1$, AC is equivalent to \emph{strict convexity} of $\cF$. In \cite{de2020equivalence}, De Rosa and Kolansinski have proven that the atomic condition implies the Almgren's strict ellipticity condition. We refer the reader to the following works about this type of functionals in higher codimension, where basic questions remain open to date: \cite{paiva2004volumes,busemann1962convex,busemann1963convex}.

Several important regularity theorems have been obtained for anisotropic minimizers. In particular, Almgren \cite{Alm3} proved regularity for sets minimizing an elliptic anisotropic energy in any dimension and codimension; Duzaar and Steffen, \cite{duzaar2002optimal}, exhibited how to obtain interior and boundary regularity for integer rectifiable currents in any dimension and codimension that almost minimize an elliptic anisotropic energy. Schoen, Simon and Almgren \cite{SSA} proved that, in codimension $1$, anisotropic energy minimizers in the sense of currents have singular set of Hausdorff codimension at least $2$; De Philippis and Maggi in \cite{philippis2015regularity} proved regularity for free boundary Caccioppoli sets that minimize an elliptic anisotropic energy. Figalli in \cite{figalli2017regularity} focused on the proof of regularity for almost minimal integral rectifiable currents, in codimension $1$ and with density $1$, under weak conditions on the anisotropic functional: namely $C^{1,1}$ anisotropies rather than the usual $C^2$ assumption. We also refer the reader to \cite{Hardt,LinPhD,DDH} for the boundary regularity of anisotropic energy (almost) minimizers and stable surfaces.

However, the regularity theory of stationary points for anisotropic integrands is much less understood, due to the number of nontrivial difficulties caused by the lack of a monotonicity formula and of mass ratio bounds. 
For codimension $1$ varifolds, Allard proved regularity under a \emph{density lower bound assumption} \cite[The basic regularity Lemma, Assumption (1)]{ARCATA}.
De Lellis, De Philippis, Kirchheim, and Tione presented in an expository fashion several open questions in the theory, see \cite{de2019geometric}. 
To the best of our knowledge, for codimension bigger than or equal to $2$, the only regularity result for varifolds that are stationary for an anisotropic energy is proved by De Rosa and Tione in \cite{de2020regularity} for varifolds induced by Lipschitz graphs, provided the anisotropic integrand satisfies the uniformly scalar atomic condition (USAC) introduced in \cite[Definition 3.3]{de2020regularity}, c.f. \cref{D:USAC}.

\subsection{Main result}

The aim of this work is to prove the anisotropic counterpart of Allard's boundary regularity theorem \cite{AllB}. To this aim, we will consider anisotropic integrands satisfying (USAC), c.f. \cref{D:USAC}. Our main result is the following. For a more precise and detailed statement, we refer to \cref{T:regularity}.
\begin{result}
Let $m,n\geq 2$, $\cF$ be an integrand of class $C^2$ satisfying USAC, $\Gamma\subset \R^{m+n}$ be an $(m-1)$-submanifold of class $C^{1,\alpha}$, $\Omega \subset \R^m$, $u\in Lip(\Omega,\R^n)$, and $\partial \mathrm{graph}(u) = \Gamma$. Assume that the anisotropic mean curvature of $u$ is in $\mathrm{L}^p$ for $p>m$. Then there exist three constants $\sigma>0$, $\beta \in (0,1)$ and $\eta \in(0,1)$ depending only on $m,n,p,\cF, \|u\|_{\mathrm{Lip}}, \Gamma$, with the following property. If $x\in\Gamma$ and $r_0>0$ are such that
$$\frac{\|\mathrm{graph}(u)\|(\ball{x}{r})}{\omega_m r^m} \leq \frac12 + \sigma \qquad \forall r\in (0,r_0),$$
then
\begin{equation*}
     u\in C^{1,\eta}(\bball{x}{\beta r_0}).
\end{equation*}
\end{result}

%

\section{Notation and preliminaries}\label{basic-concepts}

We fix integers $m,n\geq 1$ and denote $\R_+ := \{t\in\R: t\geq 0\}$. We denote by $U$ an open subset of $\R^{m+n}, \ball{x}{r}:=\{y\in\R^{m+n}:|x-y|<r\}, \oball{r}:= \ball{0}{r}$. If $\pi$ is a linear subspace of $\R^{m+n}$, we denote $\tiltball{\pi}{x}{r}:= \ball{x}{r}\cap (x+\pi),$ and we also denote $\bp_{\pi}$ the orthogonal projection from $\R^{m+n}$ onto $\pi$. When $\pi = \R^m\times \{0\}, $ we omit $\pi$ in the preceding notations.

For $s\geq 0$, $\cH^s$ denotes the $s$-dimensional Hausdorff measure induced by the Euclidean metric in $\R^{m+n}$, and $\omega_s:=\cH^s(\tiltball{\pi}{0}{1})$ where $\pi$ is an $s$-dimensional subspace. We denote the inner product of vectors by $\left\langle , \right\rangle : \R^{m+n}\times\R^{m+n}\to\R,$ the product of matrices by $\cdot$ where to any $A = ( a_{ij} )_{i=1,\ldots,h}^{j=1,\ldots,r}$ and $B = ( b_{ij} )_{i=1,\ldots,r}^{j=1,\ldots,s}$ it assigns $A\cdot B = \left( \sum_{k=1}^{r}a_{ik}b_{kj} \right)_{i=1,\ldots,h}^{j=1,\ldots,s},$ and $A:B = \mathrm{tr}(A^t\cdot B).$

For the basic theory that we will assume, we refer the reader to \cite{Fed}, \cite{simon2014introduction}, \cite{All}, and the references therein.

\subsection{Measures, rectifiability and  Grassmannian}

We denote by $\fM(U, \R^{m})$ the set of $\R^m$-valued Radon measures on $U$, when $m=1$, we denote with $\fM_+(U)$ the set of nonnegative Radon measures on $U$. Given $\mu\in\fM(U, \R^m)$, we set:
\begin{itemize}
    \item for a Borel set $A\subset U$, $\mu\res A ( E):= \mu(E\cap A)$ as the \emph{restriction of $\mu$ to $A$};
    
    \item $\|\mu \|\in\fM_+(U)$ to be the \emph{total variation of $\mu$}. Recall that, for any open set $A\subset U$, 
    \begin{equation*} \|\mu\|(A) :=  \sup\left\{ \int\left<g(x),\diff\mu(x)\right>: g\in C_c^{\infty}(A,\R^m), \|g\|_{\infty}\leq 1 \right\},\end{equation*}
    where $\left<g(x),\diff\mu(x)\right>:= \sum_{i=1}^{m}g_i(x)\mathrm{d}\mu_i(x)$;
    
    \item the \emph{upper and lower $s$-dimensional density of $\mu$ at $x$}, respectively, as
    \begin{equation*} \udens{s}{\mu}{x} := \limsup_{r\to 0^+} \frac{\|\mu\|(\ball{x}{r})}{\cH^s(\ball{x}{r})}, \quad \ldens{s}{\mu}{x} := \liminf_{r\to 0^+} \frac{\|\mu\|(\ball{x}{r})}{\cH^s(\ball{p}{r})}.\end{equation*}
    In case $\udens{s}{\mu}{x} = \ldens{s}{\mu}{x},$ we call this number the \emph{density of $\mu$ at $x$} and denote it by $\dens{s}{\mu}{x}$;
    
    \item for a Borel function $g:U\to \R^n$, the \emph{push-forward of $\mu$ through $g$} as $g_{\sharp}\mu = \mu\circ g^{-1}$.
\end{itemize}

Let $M\subset U\subset\R^{m+n}$, we say that $M$ is \emph{$s$-rectifiable} if there exist a sequence of Lipschitz maps $\{g_j:\R^s\to  U\}_{j=1}^{+\infty}$ and an $\cH^s$-null set $M_0$ such that
\begin{equation*} M = M_0 \cup\left(\bigcup_{j=1}^{+\infty}g_j(M_j)\right). \end{equation*}

In \cite[Lemma 1.2, Chapter 3]{simon2014introduction}, it is shown that $M$ is $s$-rectifiable if, and only if, $M$ can be covered, up to a $\cH^s$-null set, by countably many $s$-dimensional submanifolds of $U$ of class $C^1$. A nonnegative Radon measure $\mu\in\fM_+(U)$ is said to be \emph{$s$-rectifiable}, if there is an $s$-rectifiable set $M\subset U$ and a nonnegative Borel function $\Theta: U\to\R_+$ such that $\mu = \Theta\cH^s\res M$.

The Grassmannian of $s$-dimensional linear subspaces of $\R^{m+n}$ is denoted by $\Gr(m+n,s)$, we will often call $\pi\in\Gr(m+n,s)$ as an $s$-plane in $\R^{m+n}$. We endow $\Gr(m+n,s)$ with the metric
\begin{equation*}\|\pi - \tilde \pi\| := \sqrt{\sum_{i,j=1}^{m+n}\left(\left<\be_i, \bp_\pi(\be_j)\right> - \left<\be_i, \bp_{\tilde \pi}(\be_j)\right> \right)^2 },\quad \forall \pi, {\tilde \pi} \in\Gr(m+n,s),\end{equation*}
where $\bp_\pi$ and $\bp_{\tilde \pi}$ denote the orthogonal projections of $\R^{m+n}$ on $\pi$ and ${\tilde \pi}$, respectively, and $\{\be_i\}_{i=1}^{m+n}$ is the canonical orthonormal basis of $\R^{m+n}$. We also fix the notation 
\begin{equation*} \Gr(A,m+n,s) := A\times\Gr(m+n,s), \quad \forall A\subset U\subset \R^{m+n},\end{equation*}
and $\Gr(A):=  \Gr(A,m+n,m)$.

\subsection{Varifolds}

We say that $\bV$ is an \emph{$m$-varifold on $U$} if $\bV$ is a nonnegative Radon measure defined on $\Gr(U)$. The space of all $m$-varifolds on $U$ is denoted by $\V_m(U)$. For every $\bV\in \V_m(U)$ we can define the measure $\|\bV\|\in\fM_+(U)$, which is often called \emph{weight of $\bV$}, by the relation
\begin{equation*} \|\bV\|(A) = \bV(\proj^{-1}(A)), \quad \forall A\subset U,\end{equation*}
where henceforth $\proj$ denote the canonical projection of $\Gr(U)$ on $U$. Hence, we define  
\begin{equation*} \udens{m}{\bV}{x} :=  \udens{m}{\|\bV\|}{x}, \qquad \ldens{m}{\bV}{x} := \ldens{m}{\|\bV\|}{x}, \end{equation*}
and, when $\dens{m}{\|\bV\|}{x}$ exists,
$$\dens{m}{\bV}{x} :=  \dens{m}{\|\bV\|}{x}.$$

Of particular interest are rectifiable varifolds, which enjoy a richer structure than general varifolds, see \cite[Chapter 4 and 9]{simon2014introduction}. In fact, we say that $\bV\in\V_m(U)$ is an \emph{$m$-rectifiable varifold} if, there exists an $m$-rectifiable set $M$ in $U$ and a positive locally $\cH^m$-integrable function $\Theta$ on $M$ with $\Theta \equiv 0$ on $\R^{n}\setminus M$ such that 
\begin{equation*} \bV(A) = \int_{\proj(A)\cap M}\Theta(y)\diff\cH^m(y), \quad \forall A\subset \Gr(U). \end{equation*}

In this case, we use the notation $\bV = \bv(M,\Theta)$. For every diffeomorphism $\psi \in C^1_c(U,\R^{m+n})$, the push-forward $\psi^\#\bV\in \V_m(U)$ of $\bV\in\V_m(U)$ with respect to $\psi$ is defined as
$$\int_{\Gr(U)}\Phi(x,\pi)d(\psi^\#\bV)(x,\pi)=\int_{\Gr(U)}\Phi(\psi(x),d_x\psi(\pi))J\psi(x,\pi) d\bV(x,\pi), \; \forall \Phi\in C^0_c(\Gr(U)).$$

Here $d_x\psi(\pi)$ denotes the image of $\pi$ under the map $d_x\psi(x)$ and 
\[
J\psi(x,\pi):=\sqrt{\det\Big(\big(d_x\psi\big|_\pi\big)^*\circ d_x\psi\big|_\pi\Big)}
 \]
is the $m$-Jacobian determinant of the differential $d_x\psi$ restricted to $\pi$, see \cite[Chapter 8]{simon2014introduction}. 

We consider an \emph{anisotropic integrand} to be a $C^1$ function $\cF: \Gr(U)  \to  (0,+\infty)$ and we define the anisotropic energy of $\bV$ with respect to the anisotropic integrand $\cF$ in $A$ as 
\begin{equation*} \cE_{\bV}(A) := \int_{\Gr(A)}\cF(y,\pi)\diff\!\bV(y,\pi).\end{equation*}

Note that the area integrand is recovered when we consider $\cF \equiv 1$.

We define the notion of \emph{anisotropic first variation} or \emph{$\cF$-first variation} of an $m$-varifold $\bV$ as the distribution that acts on each $g\in C^{1}_c(U, \R^n)$ as follows
\begin{equation*} \delta_{\cF} \bV(g) := \frac{\mathrm{d}}{\mathrm{d} t} \cE_{(\phi_{t}^{\#}\bV)}(U)\biggl|\biggr._{t=0},\end{equation*}
where $\phi_t(x):= x + tg(x)$. If $\delta_{\cF}\bV \equiv 0$, we say that $\bV$ is \emph{anisotropically stationary} or \emph{$\cF$-stationary}.

We recall the following formula for the anisotropic first variation of a varifold:
\begin{proposition}[Lemma A.2, \cite{de2018rectifiability}]\label{P:1var-int}
Let $\cF \in C^{1}(\Gr(U))$ and $\bV\in \V_m(U)$, then for every $g \in C_{c}^{1}\left(U, \R^{m+n}\right)$ we have
\begin{equation*}
\delta_{\cF} \bV(g)=\int_{\Gr(U)}\biggl[\left\langle D_{x} \cF(x, \pi), g(x)\right\rangle+\cB_{\cF}(x, \pi): Dg(x)\biggr]\mathrm{d}\bV(x,\pi),
\end{equation*}
where the matrix $\cB_{\cF}(x, \pi) \in \R^{m+n} \otimes \R^{m+n}$ is uniquely defined by
\begin{equation}\label{eq1}
\begin{aligned}
\cB_{\cF}(x, \pi): L := \cF(x, \pi)(\pi: L) +\left\langle D_{\pi} \cF(x, \pi), \pi^{\perp} \circ L \circ \pi +\left(\pi^{\perp} \circ L \circ \pi\right)^{*}\right\rangle,
\end{aligned}
\end{equation}
for all $L \in \R^{m+n} \otimes \R^{m+n}.$
\end{proposition}

If we assume that $\delta_{\cF}\bV$ is a Radon measure on $\oball{r_0}\setminus \Gamma$, there exists a $\|\bV\|$-measurable function $\cH_{\cF}:\oball{r_0}\setminus \Gamma\to\R^{m+n}$ called either \emph{anisotropic mean curvature vector} or \emph{$\cF$-mean curvature vector} such that 
\begin{align}
\delta_{\cF}\bV(g) &= - \int_{\oball{r_0}\setminus \Gamma} \left\langle \cH_{\cF},g\right\rangle\mathrm{d}\|\bV\|, \quad \forall g\in C^{1}(\oball{r_0}) \text{ s.t. } g|_{\Gamma} \equiv 0, \label{E:mean-curv-def}\\
|\cH_{\cF}(x)| &= D_{\|\bV\|}\|\delta_{\cF}\bV\|(x), \quad \forall x\in\oball{r_0}\setminus \Gamma,\nonumber
\end{align}
where $D_{\|\bV\|}\|\delta_{\cF}\bV\|$ denoted the Radon-Nykodim derivative. 

\subsection{Assumptions on the anisotropic integrand}

As we briefly mentioned in the introduction, there are several ellipticity conditions which one might impose on $\cF$. We refer the reader to the references in Section \ref{Sec2}. 
We will just recall the ellipticity condition that we will use in this paper, i.e. the \emph{uniformly scalar atomic condition}, introduced in \cite[Definition 3.3]{de2020regularity}.

To this aim, we denote the dual function of $\cF$ by $\cF^*$ which is defined on $\Gr(U,m+n,n)$ as $\cF^*(x, \pi) := \cF(x, \pi^{\perp})$.

\begin{definition}[Uniformly scalar atomic condition]\label{D:USAC}
Given an anisotropic integrand $\cF \in C^{1}(\Gr(U))$, $\cF$ satisfies the uniformly scalar atomic condition (USAC) if for every $x\in U$ there exists a constant $K_{\cF,x}>0$ such that
\begin{equation*}
\cB_{\cF}(x,\pi_0) : \cB_{\cF^{*}}(x,\pi_1^{\perp}) \geq K_{\cF,x}\|{\pi_0} - {\pi_1}\|^2, \quad  \forall\pi_0, \pi_1\in\Gr(m+n,m).
\end{equation*}
\end{definition}

\begin{remark}\label{rem:rectifiability0}
We recall that De Rosa and Tione proved in \cite[Proposition 3.5]{de2020regularity} that USAC implies the so-called atomic condition. The atomic condition was in turn introduced in \cite[Definition 1.1]{de2018rectifiability} to prove the Rectifiability Theorem (\ref{R} with respect to the anisotropic first variation $\delta_{\cF}$). Hence, the Rectifiability Theorem \ref{R} holds assuming that the anisotropic integrand satisfies USAC.
\end{remark}

\section{Anisotropic first variation at boundary points}

We isolate here the assumptions under which we work in this section. 

\begin{assumption}\label{assump:general}
We set the boundary, varifold and anisotropy assumptions as follows:
\begin{enumerate}
    \item[]\label{assump:bdr} \textbf{(Boundary)} Let $\Gamma$ be a closed $(m-1)$-dimensional submanifold of class $C^{1,\alpha}$ for some $\alpha\in (0,1]$. Assume that $0\in \Gamma$, the radius $r_0>0$ is such that $\Gamma\cap\oball{r_0}$ is a graph of a $C^{1,\alpha}$ function over $T_0\Gamma$ and $\kappa\geq 0$ is a constant which satisfies
    \begin{equation}\label{E:bdr_assump} |\bp_{N_x\Gamma}(x-y)|\leq\kappa|x-y|^{1+\alpha}, \quad \| \bp_{N_x\Gamma} - \bp_{N_y\Gamma}\|\leq \kappa |x-y|^{\alpha} \quad \text{, and } \quad c\kappa r_0^{\alpha}<\frac{1}{2},\end{equation}
    for all $x,y\in \Gamma\cap \oball{r_0}$, where we use the notation $T_x\Gamma$ and $N_x\Gamma$ for the tangent and normal space to $\Gamma$ at $x$, respectively;
    
    \item[]\label{assump:var} \textbf{(Varifold)} Let $\bV\in \V_m(\oball{r_0})$ satisfying $0\in\spt(\bV)$ and $\Theta(x)\geq 1$ for $\|\bV\|$-almost every $x\in\oball{r_0}.$ We assume that $\delta_{\cF}\bV$ is a Radon measure when restricted to $\oball{r_0}\setminus \Gamma$, and the $\cF$-mean curvature $\cH_{\cF}$ of $\bV$ belongs to $\mathrm{L}^{1}\left(\oball{r_0}\setminus \Gamma, \bV\right)$;
    
    \item[]\label{assump:ani} \textbf{(Anisotropy)} Let $\cF\in C^{1}(\Gr(\oball{r_0}))$.
\end{enumerate}
\end{assumption}

\subsection{A good distance function}

If $\Gamma$ were of class $C^{1,1}$ we would have that $\Gamma$ has strictly positive reach and the distance function $\dist(x,\Gamma)$ is differentiable (not necessarily of class $C^1$) in a tubular neighborhood of thickness of the reach. However, for a $C^{1,\alpha}$ boundary $\Gamma$, the distance function is not necessarily differentiable and thus we need to ``smoothen it". Bourni in \cite[Section 3]{Theodora} showed how to properly construct this smooth distance function and we briefly recall the main properties that we are going to use in our work.  

Following the scheme of \cite[Definition 5.3.2 and 5.3.9]{krantz1999geometry}, let $\cW$ be a Whitney decomposition of $\oball{r_0}\setminus \Gamma$ into nontrivial closed $(m+n)$-cubes such that, for every $C\in\cW,$ we have that

\begin{equation*}\diam(C) \leq \dist(C,\Gamma) \leq 3\diam(C).\end{equation*}

We will fix the following notations: $x_C$ is the center of the cube $C,$ $p_C$ is a point in $\Gamma$ that satisfies $|x_C - p_C|=\dist(x_C,\Gamma)$ and $\{\varphi_C\}_{C\in\cW}$ is a Whitney partition of the unity associated to $\cW$ as in \cite[Definition 5.3.9]{krantz1999geometry} such that
\begin{equation}\label{E:grad_part_unity}
|D\varphi_C(x)|\leq \frac{c}{\dist(x, \Gamma)},
\end{equation}
where $c\geq 2$ is a dimensional constant. Since by construction $\sum_{C\in\cW}\varphi_C \equiv 1,$ and for every $x$ there exists $C_x\in\cW$ such that $\varphi_{C_x}(x)>0$, therefore
\begin{equation}\label{E:partunit_quadrado} \sum_{C\in\cW} \varphi^2_C(x) \geq C(m,n, r_0) > 0.
\end{equation}

We recall the following lemma:

\begin{lemma}[\cite{Theodora}]\label{Lgood-dist}
If we assume that $c\kappa r_0^{\alpha} < 1/2$, there exists $\rho: \oball{r_0} \to \R_+$ such that 
\begin{enumerate}[\upshape (i)]
    \item $\rho$ is a positive function of class $C^1$ with $|D\rho(x)|\leq 1 + c\kappa \rho(x)^{\alpha}$;
    
    \item the following equality holds
    \begin{equation*}\rho(x)D\rho(x) = \sum_{C\in\cW}\varphi_C(x)\bp_{N_{p_C}\Gamma}(x-p_C) + Y(x),\end{equation*}
    where $|Y(x)|\leq c\kappa\dist(x,\Gamma)^{1+\alpha} \leq c\kappa \rho(x)^{1+\alpha}$;
    
    \item we have that
    \begin{equation*}\frac{\dist(x, \Gamma)}{2} \leq\left(1-c \kappa \dist(x, \Gamma)^{\alpha}\right) \dist(x, \Gamma) \leq \rho(x) \leq\left(1+c \kappa \dist(x, \Gamma)^{\alpha}\right) \dist(x, \Gamma) \leq \frac{3 \dist(x, \Gamma)}{2}.\end{equation*}
\end{enumerate}
\end{lemma}

\begin{remark}\label{rem:whitney_barB}
Notice that, the constructions in this subsection do work if we replace $\Gamma$ by any $k$-manifold of class $C^{1,\alpha}$ with $k<m+n$.
\end{remark}

\subsection{First variation formula}

We state the formula for the anisotropic first variation at boundary points in the following proposition. First, following Allard's framework, we show that, under \cref{assump:general}, the anisotropic first variation is a Radon measure in the whole ball $\oball{r_0}$, i.e., including the boundary $\Gamma$.

\begin{proposition}\label{P:1var-bdr}
Under \cref{assump:general}, $\delta_{\cF}\bV$ is a Radon measure on $\oball{r_0}$. Moreover, there exists a $\|\delta_{\cF}\bV\|$-measurable function $\cN_{\cF}$ defined on $\Gamma$ such that $\cN_{\cF}(p)\in N_p \Gamma, \forall p\in \Gamma,$ and 
\begin{equation*} \delta_{\cF}\bV(g) = - \int_{\oball{r_0}\setminus \Gamma} \left\langle  \cH_{\cF},g\right\rangle\mathrm{d}\|\bV\| + \int_{\Gamma} \left\langle  \cN_{\cF},g\right\rangle\mathrm{d}\|\delta_{\cF}\bV\|_{\textup{sing}}, \quad \forall g\in C^{1}(\oball{r_0}).\end{equation*}
\end{proposition}
\begin{remark}\label{rem:rectifiability}
Thanks to \cref{P:1var-bdr}, under \cref{assump:general}, $\delta_{\cF}\bV$ is a Radon measure on the whole ball $\oball{r_0}$ and $\Theta\geq 1$, $\|\bV\|$-a.e. in $\oball{r_0}$. Hence, if $\cF$ satisfies USAC, by the Rectifiability criterium \cite[Theorem 1]{de2018rectifiability} and \cref{rem:rectifiability0}, the varifold $\bV$ shall be $m$-rectifiable.
\end{remark}
\begin{proof}
We want to show that for any compact subset $W\subset \oball{r_0}$ and $g$ of class $C^{1}$ with support in $W$, we have $\delta_{\cF}\bV(g) \leq C \sup_{x\in\oball{r_0}}|g(x)|$. To that end, we cannot directly apply \eqref{E:mean-curv-def}, since $g$ does not need to vanish on $\Gamma$. We thus define the family of smooth functions $f_h:\R\to\R$ such that $h\in]0,1[$,
\begin{equation*}
f_h(t) = \begin{cases}
1, & \text{if } t\leq h/2, \\
0, & \text{if } t\geq h,
\end{cases}, \quad f_h^{\prime}(t) \leq 0, \quad |f_h^{\prime}(t)| \leq 3/h.
\end{equation*}

Recalling the definition of $\rho$ in \cref{Lgood-dist}, by \cref{P:1var-int}, we obtain that
\begin{equation*}
\begin{aligned}
\delta_{\cF}\bV(g) &= \int_{\Gr(\oball{r_0}\setminus \Gamma)}\biggl[\left\langle D_x\cF, g\right\rangle + \cB_{\cF}: Dg\biggr]\mathrm{d}\bV \\
&= \overbrace{\int_{\Gr(\oball{r_0}\setminus \Gamma)}\left\langle D_x\cF, g\right\rangle\mathrm{d}\bV}^{(\ast)} + \int_{\Gr(\oball{r_0}\setminus \Gamma)}\cB_{\cF}: D\left(g+(f_h\circ\rho) g - (f_h\circ\rho) g\right)\mathrm{d}\bV.
\end{aligned}
\end{equation*}

Notice that $(\ast)$ is controlled by $C_{\cF, W} \sup_{\oball{r_0}}|g|,$ thus it remains to bound 
\begin{equation}\label{E:add-subt-f_h-rho}
    \int_{\Gr(\oball{r_0}\setminus \Gamma)}\cB_{\cF}: \biggl[ \overbrace{D\left(\left( 1 - f_h\circ\rho \right)g\right)}^{T_1}+\overbrace{\left(f_h\circ\rho \right) Dg}^{T_2} +  \overbrace{f_h^{\prime}\circ\rho (\nabla \rho)^t\cdot g}^{T_3} \biggr]\mathrm{d}\bV.
\end{equation}

Using that $\cF$ is of class $C^{1}$ and $g$ has support in $W$, by the definition of $\cB_{\cF}$ in \eqref{eq1}, we can bound the modulus of \eqref{E:add-subt-f_h-rho} by $C |T_1+T_2+T_3|,$ where the constant is such that $C = C(\cF, W)>0$. 

Since $\left( 1 - f_h\circ\rho \right)g$ vanishes on $\Gamma$, by \eqref{E:mean-curv-def}, we have that 
\begin{equation}\label{E:mean-curv-1var}
\int_{\Gr(\oball{r_0}\setminus \Gamma)}\cB_{\cF}: D(\left( 1 - f_h\circ\rho \right)g)\mathrm{d}\bV  = - \int_{\Gr(\oball{r_0}\setminus \Gamma)} \left\langle \left( 1 - f_h\circ\rho \right)g, \cH_{\cF}+D_x\cF\right\rangle\mathrm{d}\bV.
\end{equation}

We notice that $f_h\circ\rho \to 0$ as $h\to 0$, which together with \eqref{E:mean-curv-1var} ensures the estimate $|T_1|+|T_2| \leq C_1(\cF, W)\sup|g|$. It remains to bound the last summand $T_3$ by $C_2(\cF, W)\sup|g|$, which is done by precisely the same proof provided in \cite[Equation 3.10]{Theodora}. Therefore we have that $\delta_{\cF}\bV(g) \leq C \sup_{\oball{r_0}}|g|$ which guarantees that $\delta_{\cF}\bV$ is a Radon measure on $\oball{r_0}$. The moreover part can be proved as in \cite[Theorem 3.1]{Theodora}, hence we omit the details here.
\end{proof}

\section{Caccioppoli inequality at boundary points}

A usual step in the proof of regularity theorems is proving an estimate where the excess is controlled by the height, mean curvature, and an 'error' in case of 'boundary points'. This is the so-called Caccioppoli-type inequality. To the best of our knowledge, there is no such result for boundary points of $m$-rectifiable varifolds with $\mathrm{L}^2$-integrable anisotropic mean curvature. 

Allard did prove a Caccioppoli-type inequality in \cite[Lemma 4.5]{AllB} for the area functional. Unfortunately, the techniques used in the isotropic case do not work in the anisotropic case due to the lack of a monotonicity formula. We also have another difficulty compared to Allard's work: our boundary $\Gamma$ has regularity $C^{1,\alpha}$ while the setting of \cite{AllB} requires a boundary $\Gamma$ of class $C^{1,1},$ as explained in the introductory section. 

We aim to achieve a Caccioppoli-type inequality (\cref{P:caccioppoli_ineq}) in the sense of \cite[Lemma 4.5]{AllB}, \cite[Proposition 4.3]{de2020regularity}, and \cite[Lemma 4.10]{Theodora}. 

\begin{assumption}\label{assump:reg_theory}
We assume \cref{assump:general}. We further impose that the anisotropic functional $\cF$ satisfies USAC, defined in \cref{D:USAC}, and $\cH_{\cF}\in\mathrm{L}^{2}(\oball{r_0})$.
\end{assumption}

Under such assumptions, by \cref{rem:rectifiability}, the varifold $\bV$ is $m$-rectifiable. So, henceforth we might use the following notation $\bV = \bv(M,\Theta)$. We define the classical notions of excess and height for varifolds as follows.

\begin{definition}\label{D:excess_height}
Let $\bV = \bv(M,\Theta)$ be a rectifiable $m$-varifold and $\pi\in\Gr(m+n,m).$ We define \emph{the tilt excess of $\bV$ with respect to $\pi$ in $\ball{x}{r}$} as the number
\begin{equation*}\bE_{\bV}(\pi, x, r):= \frac{1}{r^{m}}\int_{\ball{x}{r}}\|{\pi} - {T_yM}\|^2 \mathrm{d}\|\bV\|(y).  \end{equation*}

We also define \emph{the height excess of $\bV$ with respect to $\pi$ in $\ball{x}{r}$} to be the number
\begin{equation*} \bH_{\bV}(z, \pi, x, r) := \frac{1}{r^{m}} \int_{\ball{x}{r}} \dist(y-z,\pi)^2\mathrm{d}\|\bV\|(y) .\end{equation*}

We usually hide the subscripts  whenever it is clear from the context.
\end{definition}

We now state the Caccioppoli-type inequality in this context.

\begin{proposition}[Caccioppoli-type inequality]\label{P:caccioppoli_ineq}
Under \cref{assump:reg_theory}, there exists a constant $C=C(m,n,\|\cF\|_{C^2},K_{\cF}, \Gamma)>0$ such that 
\begin{equation}\label{cacc0}
C\bE(\pi, 0, r/2) \leq \frac{1}{r^2}\bH(z, \pi,0,r) + r^{2-m}\|\cH_{\cF}\|_{\mathrm{L}^{2}(\oball{r})}^2 + \kappa^2r^{2\alpha},
\end{equation}
for all $z\in\R^{m+n}, 4r<r_0, \pi\in\Gr(m+n,m)$ with $T_0\Gamma\subset\pi$.
\end{proposition}

When the varifold $\bV$ is induced by the graph of a Lipschitz function, the next corollary states that the quantities in \cref{P:caccioppoli_ineq} can be replaced by integrations on balls of the subspace $\R^{m}$, while in \cref{P:caccioppoli_ineq} they are quantities/integrations over balls of the ambient space $\R^{m+n}.$ 

Given an open bounded set $\Omega\subset\obball{r_0}\subset\R^m $ and a Lipschitz function $u: \Omega\to \R^n$, we will denote by 
$$\bV[u]:=\bv(\graph(u),1)$$
 the $m$-varifold induced by $\graph(u)\subset\R^{m+n}$ and by $\cH_{\cF}$ its anisotropic mean curvature. Let also $\cH[u]:\Omega\to\R^n$ denote the function $\cH[u](x):=\cH_{\cF}(x,u(x))$ and, for any $R>0, z\in\obball{R}, s<\dist(z,\partial\obball{R}),$ and $f:\obball{R}\subset\R^m\to\R^n$ measurable function, we set
\begin{equation*} (f)_{x,s}:= \frac{1}{\cH^m\left(\bball{x}{s}\cap s\Omega\right)}\int_{\bball{x}{s}\cap s\Omega}f(y)\mathrm{d}y \quad \text{and} \quad (f)_{s}:=(f)_{0,s}.\end{equation*}

For the reader's convenience, we recall that $\bball{x}{s}:=\ball{(x,0)}{s}\cap(\R^m\times\{0\}).$

\begin{corollary}[Caccioppoli-type inequality]\label{cor:caccioppoli_ineq_graph}
Assume that $\bV[u]$ and $\cH[u]$ satisfy \cref{assump:reg_theory}. There exists a constant $C_c=C_c(m,n,\|\cF\|_{C^2},K_{\cF},\Gamma, \|u\|_{\mathrm{Lip}})>0$ and $C_u:= 2 + 2 \|u\|_{\mathrm{Lip}}$ such that 
\begin{equation*}
\begin{aligned}
C_c\mint_{\obball{r}\cap\Omega}\|Du(y) - L\|^2\mathrm{d}y &\leq \frac{1}{r^{2}} \mint_{\obball{C_u r}\cap\Omega}|u(y) - L(y)|^2\mathrm{d}y \\
&\quad + r^2\mint_{\obball{C_u r}\cap\Omega}|\cH[u](y)|^2\mathrm{d}y + \kappa^2r^{2\alpha},
\end{aligned}
\end{equation*}
for all $L\in\R^{m}\otimes\R^n$ such that $T_0\Gamma\subset\mathrm{im}(h(L))$ and $\|L\|\leq 2\|u\|_{\mathrm{Lip}}$ and all $r\in(0,4^{-1}r_0)$.
\end{corollary}
\begin{remark}\label{R:def-h}
The function $h$ stands for one of the canonical charts of the Grassmannian, we make it precise defining $h:\R^m\otimes\R^n\to \R^{m+n}\otimes\R^{m+n}$ as
\begin{equation*}
    h(L) := M(L)\left[M(L)^tM(L)\right]^{-1}M(L)^t, \text{ where } M(L):= \biggl( \begin{array}{c} \mathrm{id}_m \\ L \end{array} \biggr). 
\end{equation*}    

We refer the reader to \cite[Subsection 6.1]{de2019geometric}, \cite[Page 470]{de2020regularity}, and \cite[Subsection A.6]{hirsch2021constancy} for a more expository introduction to these objects.
\end{remark}
\begin{proof}
Extending this proof from the interior case to boundary points setting is identical to the argument presented in \cite[Corollary 4.4]{de2020regularity}, but now relying on \cref{P:caccioppoli_ineq}. We point out that one may choose $p$ equal to $(0,0)\in\R^m\times\R^n$ in the proof of \cite[Corollary 4.4]{de2020regularity} instead of $p=(0, (u)_{0,s})$ to get rid of the mean of the function in the right-hand side. This is surely possible thanks to the fact that \cref{P:caccioppoli_ineq} holds true for any choice of point $z\in \R^{m+n}$.
\end{proof}

\begin{proof}[Proof of \cref{P:caccioppoli_ineq}]
First of all, by standard arguments, cf. \cite[Page 465]{de2020regularity}, we can assume without loss of generality that $\cF$ is an autonomous functional, i.e., it does not depend on the variable in $\R^{m+n}$. Hence we will denote $\cB_{\cF}(\pi)\equiv \cB_{\cF}(x,\pi)$ and $K_{\cF}\equiv K_{\cF,x}$.
We can set the $m$-manifold of class $C^{1,\alpha}$ given by  $\overline{\Gamma} = \Gamma + (N_0\Gamma\cap\pi)$. In particular, by \cref{rem:whitney_barB}, we have a Whitney decomposition $\overline{\cW}$ of $\oball{r_0}\setminus \overline{\Gamma}$. We denote with $x_C$ and  $\overline{x}_C$ respectively the center of the cube $C$ and the orthogonal projection of $x_C$ on $\overline{\Gamma}$. We consider a Whithney's partition of unity $\{\overline{\varphi}_C\}_{C\in\cW}$, and a $C^1$ function $\overline{\rho}$ satisfying all the conclusions of \cref{Lgood-dist}.

We choose the following vector field $g\in C^1_c(\oball{2r}, \R^{m+n})$ as a test for the first variation:
\begin{equation*}
g(x) := \psi^2(x)\sum_{C\in\cW}\overline{\varphi}^2_C(x) g_C(x), \quad \mbox{where} \quad g_C(x) := \cB_{\cF^*}(\pi^{\perp})(\bp_{N_{\overline{x}_C}\overline{\Gamma}}(x - \overline{x}_C)),
\end{equation*}
where $\psi\in C_c^{\infty}(\oball{2r}, [0,1])$ such that $\psi |_{\oball{r}}\equiv 1$. It is important to choose $g$ using the Whitney decomposition, since it ensures that $g|_{\overline{\Gamma}}\equiv 0$, in particular $g|_\Gamma\equiv 0$, and then \eqref{E:mean-curv-def} holds. By direct computations we obtain that
\begin{align}
Dg &= \sum_{C\in\cW}\biggl[2\psi\overline{\varphi}^2_C\left(g_C\right)\cdot\left(\nabla \psi\right)^{t} + \psi^2\overline{\varphi}^2_C Dg_C + 2\psi^2 \overline{\varphi}_Cg_C\cdot \left(\nabla\overline{\varphi}_C\right)^{t}\biggr], \label{E:formula_Dg} \\
Dg_C &= \cB_{\cF^*}(\pi^{\perp})\circ\bp_{N_{\overline{x}_C}\overline{\Gamma}}. \label{E:formula_DgC}
\end{align}

Equation \eqref{E:formula_Dg} together with \eqref{E:mean-curv-def} assures that
\begin{align}\label{intermediate}
-\int \left\langle \cH_{\cF}, g\right\rangle = \int \sum_{C\in\cW}\cB_{\cF}:\biggl[2\psi\overline{\varphi}^2_C\left(g_C\right)\cdot\left(\nabla \psi\right)^{t} + \psi^2\overline{\varphi}^2_C Dg_C + 2\psi^2\overline{\varphi}_Cg_C\cdot \left(\nabla\overline{\varphi}_C\right)^{t}\biggr].
\end{align}

We set the following notation
\begin{align*}
R_1 &:= \int\sum_{C\in\cW} \psi^2(x)\overline{\varphi}^2_C(x)\left\langle \cH_{\cF}(x), g_C(x) \right\rangle\mathrm{d}\|\bV\|(x),\\
R_2 &:= \int\sum_{C\in\cW} 2\psi(x)\overline{\varphi}^2_C(x)\cB_{\cF}(T_xM):\left(g_C(x)\cdot\nabla \psi(x)^{t}\right) \mathrm{d}\|\bV\|(x),\\
R_3 &:= \int\sum_{C\in\cW}2 \psi^2(x)\overline{\varphi}_C(x)\cB_{\cF}(T_xM):\left(g_C(x)\cdot \nabla\overline{\varphi}_C(x)^{t}\right)\mathrm{d}\|\bV\|(x),\\
L_1 &:=  - \int\sum_{C\in\cW} \psi^2(x)\overline{\varphi}^2_C(x) \cB_{\cF}(T_xM) : \cB_{\cF^*}(\pi^{\perp})\circ\bp_{N_{\overline{x}_C}\overline{\Gamma}} \mathrm{d}\|\bV\|(x).
\end{align*}

By \eqref{intermediate} and \eqref{E:formula_DgC} we obtain that
\begin{equation}\label{somma}
 L_1  =  R_1 + R_2 + R_3.  
\end{equation}

We estimate $|L_1|$ from below. By the definition of $L_1$ and the uniformly scalar atomic condition, \cref{D:USAC}, recalling that $\psi |_{\oball{r}}\equiv 1$,we get
\begin{equation}\label{utile}
\begin{aligned}
| L_1 | &\geq K_{\cF}\int\sum_{C\in\cW}\psi^2(x)\overline{\varphi}^2_C(x) \|{T_xM} - {\pi}\|^2 \mathrm{d}\|\bV\|(x) \\
\overset{\eqref{E:partunit_quadrado}}&{\geq} K_{\cF}C\int_{\oball{r}}\|{T_xM} - {\pi}\|^2 \mathrm{d}\|\bV\|(x) = K_{\cF}Cr^m\bE(\pi,0,r),
\end{aligned}
\end{equation}
where here and in the rest of this proof $C=C(m,n, r_0)>0$ is defined in \eqref{E:partunit_quadrado}.
The right-hand side of \eqref{utile} is exactly the desired left hand side in the Caccioppoli-type inequality \eqref{cacc0}, up to the factor $r^m$. Therefore, it remains to bound $|R_1|+|R_2|+|R_3|$ from above with the right hand side in \eqref{cacc0} (again up to the factor $r^m$) plus a term that can be reabsorbed in the left hand side of \eqref{cacc0}. With this aim in mind, let us estimate the term $R_3$. We have that
\begin{equation*} 
    R_3 = \int\sum_{C\in\cW}2\psi^2(x)\overline{\varphi}_C(x)\cB_{\cF}(T_xM):\left(g_C(x)\cdot \nabla\overline{\varphi}_C(x)^{t}\right)\mathrm{d}\|\bV\|(x). 
\end{equation*}

By straightforward linear algebra computations, we have that
\begin{equation}\label{E:lin_alg}
    \cB_{\cF}(\pi)^{t}\cdot\cB_{\cF^*}(\pi^{\perp}) = 0
\end{equation}
which in turn implies
\begin{equation*}
\begin{aligned}
    R_3 = \int\sum_{C\in\cW}&2\psi^2(x)\overline{\varphi}_C(x)\left(\cB_{\cF}(T_xM) - \cB_{\cF}(\pi) \right):\left(g_C(x)\cdot \nabla\overline{\varphi}_C(x)^{t}\right)\mathrm{d}\|\bV\|(x).
\end{aligned} \end{equation*}

We apply Young's inequality to obtain
\begin{align}
|R_3| &\leq \frac{K_{\cF}C}{4}\int\psi^4(x)\|T_xM - \pi\|^2\mathrm{d}\|\bV\|(x) \label{E:first_bound_R3}\\
&\quad + c(m,n,K_{\cF})\int\left|\sum_{C\in\cW}\overline{\varphi}_C(x)g_C(x)\cdot( \nabla\overline{\varphi}_C(x))^t\right|^2\mathrm{d}\|\bV\|(x) \nonumber .
\end{align}

To bound the second summand on the right hand side of the last inequality, we proceed as follows
\begin{equation*}\begin{aligned}
\sum_{C\in\cW}\bp_{N_{\overline{x}_C}\overline{\Gamma}}(x - \overline{x}_C)\cdot( \nabla\overline{\varphi}_C(x))^t &= \sum_{C\in\cW}(\bp_{N_{\overline{x}_C}\overline{\Gamma}}(x - \overline{x}_C) -\bp_{N_{\overline{x}_C}\overline{\Gamma}}(x-\overline{x}))\cdot( \nabla\overline{\varphi}_C(x))^t\\
&= \sum_{C\in\cW}\bp_{N_{\overline{x}_C}\overline{\Gamma}}(\overline{x} - \overline{x}_C)\cdot( \nabla\overline{\varphi}_C(x))^t,
\end{aligned}\end{equation*}
where in the first equality we have used that $\sum_{C\in\cW}\nabla\overline{\varphi}_C\equiv 0$. Plugging the equality above in  \eqref{E:first_bound_R3}, we get that 
\begin{equation}\label{E:R3_final}
\begin{aligned}
    |R_3| &\leq \frac{K_{\cF}C}{4}r^m\bE(\pi,0,r) + c\int\sum_{C\in\cW}\overline{\varphi}_C(x)|\bp_{N_{\overline{x}_C}\overline{\Gamma}}(\overline{x} - \overline{x}_C)|^2| \nabla\overline{\varphi}_C(x)|^2\mathrm{d}\|\bV\|(x) \\
    \overset{\eqref{E:grad_part_unity}}&{\leq} \frac{K_{\cF}C}{4}r^m\bE(\pi,0,r) + c_1\int\sum_{C\in\cW}\overline{\varphi}_C(x)\frac{|\bp_{N_{\overline{x}_C}\overline{\Gamma}}(\overline{x} - \overline{x}_C)|^2}{\diam^2(C)}\mathrm{d}\|\bV\|(x)\\
    \overset{\eqref{E:bdr_assump}}&{\leq} \frac{K_{\cF}C}{4}r^m\bE(\pi,0,r) + \kappa^2 c_1\int\sum_{C\in\cW}\overline{\varphi}_C(x)\frac{|\overline{x} - \overline{x}_C|^{2+2\alpha}}{\diam^2(C)}\mathrm{d}\|\bV\|(x)\\
    &\leq \frac{K_{\cF}C}{4}r^m\bE(\pi,0,r) + \kappa^2 c_1\int\sum_{C\in\cW}\overline{\varphi}_C(x)|\overline{x} - \overline{x}_C|^{2\alpha}\mathrm{d}\|\bV\|(x)\\
    &\leq \frac{K_{\cF}C}{4}r^m\bE(\pi,0,r) + \kappa^2r^{m+2\alpha} c_1,
\end{aligned}
\end{equation}
where $c_1=c_1(m,n,\|\cF\|_{C^2},\cW)>0$.

Turning our attention to $R_1$, thanks to the hypothesis that $\cH_{\cF}$ belongs to $\mathrm{L}^2$, we apply Young's inequality and Jensen inequality to get
\begin{equation}\label{E:bound_R1}
\begin{aligned}
    |R_1|  &=  \biggl|\int \psi^2(x)\left\langle \cH_{\cF}(x),\sum_{C\in\cW}\overline{\varphi}^2_C(x) g_C(x) \right\rangle\mathrm{d}\|\bV\|(x)\biggr|\\
    &\leq r^2\|\cH_{\cF}\|_{\mathrm{L}^2(\oball{2r})}^2 + \frac{C_{1}(m,n)}{r^2} \int\sum_{C\in\cW} \psi^4(x)\overline{\varphi}^4_C(x)|g_C(x)|^2\mathrm{d}\|\bV\|\\
    &\leq r^2\|\cH_{\cF}\|_{\mathrm{L}^2(\oball{2r})}^2 + \frac{C_{1}(m,n)}{r^2} \int\sum_{C\in\cW} \overline{\varphi}_C(x)|g_C(x)|^2\mathrm{d}\|\bV\|.
\end{aligned}
\end{equation}

We now use \eqref{E:lin_alg} to estimate the summand $R_2$ as follows
\begin{equation*}
\begin{aligned}
    |R_2| &\leq \biggl|\int\sum_{C\in\cW}2\psi(x)\overline{\varphi}^2_C(x)\left(\cB_{\cF}(T_xM) - \cB_{\cF}(\pi) \right):\left(g_C(x)\cdot \nabla\psi(x)^{t}\right)\mathrm{d}\|\bV\|(x)\biggr|\\
    &\leq 2 \int \|\psi\| \| \nabla\psi \|  \|\cB_{\cF}(T_xM) - \cB_{\cF}(\pi)  \|\sum_{C\in\cW}\overline{\varphi}^2_C(x)|g_C(x)|\mathrm{d}\|\bV\|(x) \\
    &{\leq} C_2(m,n,\|\cF\|_{C^2}) \int \|\psi\| \|{T_xM}-{\pi}  \|\sum_{C\in\cW}\overline{\varphi}^2_C(x) |g_C(x) |\mathrm{d}\|\bV\|(x)\\
    &{\leq} \frac{r^2K_{\cF}C}{4} \int_{\oball{2r}} \|\psi\|^2\|{T_xM}-{\pi}  \|^2 + \frac{C_2(m,n,\|\cF\|_{C^2})}{r^2}\sum_{C\in\cW}\overline{\varphi}^2_C(x) |g_C(x) |^2\mathrm{d}\|\bV\|(x),
\end{aligned}
\end{equation*}
where in the third inequality we have used that $\cF$ is $C^2$ and that the Grassmannian is compact, and in the fourth inequality we have used again Young's inequality. Since the last chain of inequalities is true for any $\psi$ choosen as above, we can take a sequence $\{\psi_i\}_{i\in\N}\subset C_c^{\infty}(\oball{2r} [0,1])$ such that $\psi_i$ converges to the indicator functions of $\oball{r}$. Therefore we obtain that 
\begin{equation}\label{E:R2_bound}
\begin{aligned}
    |R_2| &\leq \frac{K_{\cF}C}{4} \int_{\oball{r}} \|{T_xM}-{\pi}  \|^2 + \frac{C_2}{r^2} \int_{\oball{2r}}\sum_{C\in\cW}\overline{\varphi}^4_C(x) |g_C(x) |^2\\
    &= \frac{K_{\cF}C}{4}r^m\bE(\pi,0,r) + \frac{C_2}{r^2} \int_{\oball{2r}}\sum_{C\in\cW}\overline{\varphi}_C(x) |g_C(x) |^2,
\end{aligned}
\end{equation}
where $C_2 = C_2(m,n,\|\cF\|_{C^2})>0$. We finally use  \eqref{E:R3_final}, \eqref{E:bound_R1}, and \eqref{E:R2_bound} to estimate
\begin{equation}\label{E:R1R2R3}
\begin{aligned}
    |R_1|+|R_2|+|R_3| &\leq \frac{K_{\cF}C}{2}r^m\bE(\pi,0,r) + r^2\|\cH_{\cF}\|_{\mathrm{L}^2(\oball{2r})}^2 + \kappa^2r^{m+2\alpha} c_1\\
    &\quad + \frac{C_{2}}{r^2} \int_{\oball{2r}}\sum_{C\in\cW}\overline{\varphi}_C(x)|\cB_{\cF^*}(\pi^{\perp})(\bp_{N_{\overline{x}_C}\overline{\Gamma}}(x - \overline{x}_C))|^2\mathrm{d}\|\bV\|,
\end{aligned}
\end{equation}
where $c_1=c_1(m,n,\|\cF\|_{C^2},\cW)>0$, and $C_{2}=C_{2}(m,n,\|\cF\|_{C^2})>0$. It only remains to bound the last summand of the previous inequality. We firstly recall the equality in \cite[Equation 3.5]{de2020regularity} which states that $\cB_{\cF}(\pi^{\perp}) = \cF(\pi)\pi^{\perp} - \pi D\cF(\pi)\pi^{\perp}$ and thus we obtain the following
\begin{equation*} 
\begin{aligned}
    |\cB_{\cF^*}(\pi^{\perp})(\bp_{N_{\overline{x}_C}\overline{\Gamma}}(x - \overline{x}_C))| &\leq \|\cF(\pi)\pi^{\perp} - \pi D\cF(\pi)\pi^{\perp} \| | \bp_{N_{\overline{x}_C}\overline{\Gamma}}(x - \overline{x}_C) | \\
    &\leq \|\cF\|_{C^2} | \bp_{N_{\overline{x}_C}\overline{\Gamma}}(x) - \bp_{N_{\overline{x}_C}\overline{\Gamma}}(\overline{x}_C) |\\
    &\leq \|\cF\|_{C^2} \left(| \bp_{N_{\overline{x}_C}\overline{\Gamma}}(x)| +| \bp_{N_{\overline{x}_C}\overline{\Gamma}}(\overline{x}_C) |\right)\\
    \overset{\eqref{E:bdr_assump}}&{\leq} \|\cF\|_{C^2} \left(| \bp_{N_{\overline{x}_C}\overline{\Gamma}}(x)| +\kappa|\overline{x}_C|^{1+\alpha} \right)\\
    &\leq \|\cF\|_{C^2} \left(| (\bp_{N_{\overline{x}_C}\overline{\Gamma}} - \bp_{\pi^{\perp}})(x)| + |\bp_{\pi^{\perp}}(x) | +\kappa r^{1+\alpha} \right)\\
    &\leq \|\cF\|_{C^2} \left(| (\bp_{N_{\overline{x}_C}\overline{\Gamma}} - \bp_{\pi^{\perp}})(x)| + \dist(x,\pi) +\kappa r^{1+\alpha} \right)\\
    \overset{\eqref{E:bdr_assump}}&{\leq} \|\cF\|_{C^2} \left( \kappa |\overline{x}_C|^{\alpha}|x| + \dist(x,\pi) +\kappa r^{1+\alpha} \right)\\
    &\leq 4\|\cF\|_{C^2}\left(\dist(x,\pi) + \kappa r^{1+\alpha}\right).
\end{aligned}
\end{equation*}

The chain of inequalities above with \eqref{E:R1R2R3} provides the following estimate
\begin{equation*}\begin{aligned}
|R_1|+|R_2|+|R_3| &\leq \frac{K_{\cF}C}{2}r^m\bE(\pi,0,r) + r^2\|\cH_{\cF}\|_{\mathrm{L}^2(\oball{2r})}^2  \\
&\quad +c_2(m,n,\|\cF\|_{C^2},\cW)\left(\kappa^2r^{m+2\alpha} + \int_{\oball{2r}}\frac{\dist^2(x,\pi)}{r^2}\mathrm{d}\|\bV\|\right).
\end{aligned}\end{equation*}

Combining this inequality with \eqref{utile}, and recalling \eqref{somma}, we can reabsorb $\frac{K_{\cF}C}{2}r^m\bE(\pi,0,r)$ on the left hand side and conclude the proof of \eqref{cacc0}.
\end{proof}

\section{Excess decay at boundary points}

In this section, we will work under the following \cref{assump:H_u-L_p}. It is clear that \cref{assump:H_u-L_p} is more restrictive than \cref{assump:reg_theory}.

\begin{assumption}\label{assump:H_u-L_p}
We assume \cref{assump:general}. Additionally, $\bV = \bV[u]$ and $\cH_{\cF} = \cH_{\cF}[u]$, where $u:\Omega\subset\obball{r_0}\subset\R^m\to\R^n$ is a Lipschitz function with $\cH_{\cF}[u]\in\mathrm{L}^{p}(\Omega)$ for some $p>m$. We also set $\Gamma\cap\oball{r_0} = \partial(\graph(u))\cap\oball{r_0}$ and $\partial \Omega = \bp(\Gamma)$ splits $\obball{r_0}$ into two disjoint open sets, namely $\Omega$ and $\obball{r_0}\setminus\Omega$. Moreover, there exists $\sigma\in (0,1)$ such that for every $r\in(0,r_0)$, we have
\begin{equation}\label{massratio}
    \frac{\|\bV[u]\|(\oball{r})}{\omega_m r^m} \leq \frac{1}{2} + \sigma.
\end{equation}
\end{assumption}

We recall a lemma that relates the stationarity of the function $u$ with the stationarity of the varifold $\bV[u]$ induced by $u$. This lemma is proved in \cite{de2019geometric} for the case of interior points. Let us set the notation to state it:
\begin{equation}\label{AandI}
\begin{aligned}
\cA(L) & :=\sqrt{M(L)^tM(L)},\quad  \cI_{\cF}(L) & :=\cA(L) \cF\left(h(L)\right) \qquad \forall L\in\R^m\otimes\R^n.
\end{aligned}
\end{equation}
where $h$ and $M(L)$ are defined in \cref{R:def-h}.

\begin{lemma}\label{L:technical}
Assume \cref{assump:H_u-L_p}. If for some positive constants $C$ and $q\geq 1$ it holds 
\begin{equation*}
    |\delta_{\cF}\bV[u](g)| \leq C \|g\|_{\mathrm{L}^q(\obball{r_0}\times\R^m)}, \quad \forall g\in C^{1}_c\left(\obball{r_0}\times\R^m,\R^{m+n}\right) \text{ with } g|_{\Gamma}\equiv 0,
\end{equation*} 
then there exists $C^{\prime}=C^{\prime}(C, m, p, q)>0$ such that
\begin{equation}\label{useful0}
    \left | \int_{\Omega}\langle D\cI_{\cF}(Du),D\zeta \rangle \mathrm{d}\cH^m\right| \leq C^{\prime} \|\zeta \cA^{\frac{1}{q}}(Du)\|_{\mathrm{L}^q(\obball{r_0})}, \quad \forall \zeta\in C^{1}_c\left(\obball{r_0}, \R^n\right) \text{ with } \zeta|_{\bp(\Gamma)}\equiv 0.
\end{equation}

Moreover, if $C=0$, thus $C^{\prime}=0$.
\end{lemma}
The proof of \cref{L:technical} is a straightforward extension of \cite[Proposition 6.8]{de2019geometric} to our boundary setting. Furthermore, \cite[Proposition 6.8]{de2019geometric} can be adapted to give the equivalence between the two properties. However, we choose to state only the exact statement we will use.

We now use the mass ratio bound \eqref{massratio} in \cref{assump:H_u-L_p} to prove the following technical lemma, that will allow us to apply the Caccioppoli inequality (\cref{P:caccioppoli_ineq}). 
\begin{lemma}\label{L:density}
Under \cref{assump:H_u-L_p}, there exists $C_d = C_d(m,n,\alpha)>0$, $c_0=c_0(m,n,\alpha)>0$ and $L_u\in\R^m\otimes\R^n$ with $T_0\Gamma\subset\mathrm{im}(h(L_u))$ such that $\|L_u -(Du)_r\|\leq C_dr^{\alpha} + C_d\sigma$ for any $r\in(0,c_0)$.
\end{lemma}
\begin{proof}
Without loss of generality, we can assume that $\Gamma = \oball{r_0}\cap\R^{m-1}\times\{0\}$ by a standard procedure of straightening the boundary (for instance, using \cite[Lemma 3.1]{DNS2}). By the Taylor expansion of the mass (c.f. \cite{DS1}), we obtain that
\begin{equation*}
    C_0r^{m+\alpha} + 2\left(\|\bV[u]\|(\oball{r}) - \frac{\omega_m r^m}{2}\right) \geq  \int_{\{x_m\geq0\}\cap\obball{r}}\|Du\|^2.
\end{equation*}

Thus the control over the mass ratio enables us to straightforwardly derive that
\begin{equation*}
    \|(Du)_r\| \leq C_0 r^{\alpha} + 2\sigma.
\end{equation*}

We choose $L_u:=\lim_{s\to 0}(Du)_s$ which, by the last inequality, satisfies the desired inequality. It is easy to see that $T_0\Gamma = \R^{m-1}\times \{0\}\subset\mathrm{im}(h(L_u))$, since $Du(x) = x_mv_0$ for any $x=(x^{\prime},x_m)\in\R^{m-1}\times \{0\} $ and a fixed $v_0\in\R^n$.
\end{proof}

One of the crucial parts of the regularity theory is to prove an excess decay with a precise rate of decay. Let us fix the following shorthand notation for the excess of the function $u$:
\begin{equation}\label{E:def-excess}
\begin{aligned}
    E(x,r,L):=\mint_{\bball{x}{r}\cap\Omega}\|Du(z) - L  \|^2\mathrm{d}z, \\
    E(x,r):=E(x,r,\left(Du\right)_{r}), \text{ and } E(r):= E(0,r).    
\end{aligned}
\end{equation}

We now prove the excess decay at boundary points for the function $u$, i.e., we prove that the derivative $Du$ of $u$ becomes closer in $\mathrm{L}^2$-norm to a linear map as we decrease the radius of balls centered at the origin. The proof follows a similar argument as the one for \cite[Proposition 4.5]{de2020regularity}.

\begin{proposition}[Excess decay]\label{P:excess-decay}
Under \cref{assump:H_u-L_p}, there exists a positive constant $C_e=C_e(m,n,\|\cF\|_{C^2}, K_{\cF}, \|u\|_{\mathrm{Lip}}, \Gamma)>0$ with the following property. For every $\varepsilon \in (0, 4^{-1}C_u^{-1})$, there exist $\delta_*=\delta_*(\varepsilon)>0$ such that, for every  $\delta \in (0,\delta_*)$, $r>0$ and $L\in \R^m\otimes\R^n$ satisfying $T_0\Gamma\subset \mathrm{im}(h(L))$ and 
\begin{equation}\label{E:excess_decay:assump}
\begin{aligned}
    r^{\min\{\alpha,1-\frac{m}{p}\}}\|\cH_{\cF}[u]\|_{\mathrm{L^p}(\Omega\times\R^n)} \leq E(r,L)\leq \delta,
\end{aligned}
\end{equation}
then there exists $\tilde{L}\in\R^m\otimes\R^n$ with $T_0\Gamma\subset \mathrm{im}(h(\tilde{L}))$ and $\| L - \tilde{L}\| \leq C_e\delta$ such that
\begin{equation}\label{E:excess_decay}
   E(\varepsilon r, \tilde{L}) \leq C_e \varepsilon^{2\alpha} E(r, L).
\end{equation}
\end{proposition}
\begin{proof}

As in the proof of \cref{P:caccioppoli_ineq}, we can again assume without loss of generality that $\cF$ is an autonomous functional.

We prove our statement by a contradiction argument. Assume that for every $C_e >0$ there exist $\varepsilon\in (0, 4^{-1}C_u^{-1})$ such that, for any $j\in \N$, there exists $L_j\in\R^m\otimes\R^n$ and $r_j>0$ satisfying
\begin{equation}\label{E:excess_decay:contradiction_assump1}
\begin{aligned}
    r_j^{\min\{\alpha,1-\frac{m}{p}\}}\|\cH_{\cF}[u]\|_{\mathrm{L^p}(\Omega\times\R^n)} \leq E(r_j,L) \leq \frac 1j,
\end{aligned}
\end{equation}
and, for every $\tilde{L}\in\R^m\otimes\R^n$ with $T_0\Gamma\subset \mathrm{im}(h(\tilde{L}))$ and $\| L_j - \tilde{L}\| \leq \frac {C_e}j$, it holds
\begin{equation}\label{E:excess_decay:contradiction_assump2}
    E(\varepsilon r_j, \tilde{L}) > C_e \varepsilon^{2\alpha} E(r_j,L_j).
\end{equation}
For every $j$, we choose $\tilde{L}$ to be a specific $\tilde{L}_j$, to be fixed later in \eqref{Ltilde}.

We divide our proof into three steps. In Step 1, we prove that a certain blowup sequence for $u$ weakly converges in $\mathrm{W}^{1,2}(\obball{r}\cap\Omega)$ to a limit function $u_0$. After that, we show in Step 2 that $u_0$ is a weak solution of an elliptic system of PDEs, subsequently we use regularity theory for elliptic PDEs to obtain an estimate for the second derivative of $u_0$. We close our argument in Step 3, where we apply \cref{cor:caccioppoli_ineq_graph}, together with the elliptic estimates from Step 2, to get a contradiction with \eqref{E:excess_decay:contradiction_assump2}.

\textbf{Step 1:} By \eqref{E:excess_decay:contradiction_assump1}, we can assume $r_j\to 0$. We choose $\delta_j^2 := E(r_j,L_j)$. We set $\Omega_j:= r^{-1}_j \Omega$, $\Gamma_j:=r^{-1}_j\Gamma$, and the blowup sequence as follows
\begin{align*}
u_j \colon \Omega_j &\to \R^n \\ 
z &\mapsto \frac{u(r_j z) - r_j L_jz}{\delta_j r_j}.
\end{align*}

We assume that $\delta_j>0$, otherwise there is nothing to prove. It is easy to see that $\Omega_j \to \{x\in\R^m:x_m \geq 0\}$ and $\Gamma_j \to \R^{m-1}$ as $j$ goes to $+\infty$. Furthermore, we list some properties of the sequence $u_j$ that will be used in this proof. They are:
\begin{enumerate}[\upshape (a)]
    \item\label{(a)} $Du_j(z) = \delta_j^{-1}\left( Du(r_jz) - L_j\right)$, which is a trivial computation;
    
    \item\label{(b)} $(Du_j)_1 = \delta_j^{-1}((Du)_{r_j} - L_j)$, which is a straightforward consequence of \cref{(a)};
    
    \item\label{(c)} $\mint_{\obball{1}\cap\Omega_j}\|Du_j\|^2 = \delta_j^{-2}E(r_j,L_j) = 1$, which follows changing variables and using \cref{(a)};
    
    \item\label{(d)} $\int_{\obball{1}\cap\Omega_j}\|u_j - (u_j)_1\|^2$ is uniformly bounded. This follows from the Poincar\`e-Wirtinger inequality and \cref{(c)};
    
    \item\label{(e)} $E_{u_j}(\varepsilon, \delta_j^{-1}(L_j - \tilde{L}_j)) > C_e \varepsilon^{2\alpha}E_{u_j}(1, 0) = C_e\varepsilon^{2\alpha}\delta_j^2$, where we set $E_{u_j}(r, L)$ to be the excess $E_{u_j}(r, L):=\mint_{\obball{r}\cap\Omega_j}\|Du_j(z) - L \|^2\mathrm{d}z$. This item follows from \eqref{E:excess_decay:contradiction_assump2}, \cref{(a)}, and the definition of $u_j$.
\end{enumerate}

For every $s>0$, denote the halfball $\mathrm{B}^+_s:= \obball{s}\cap \{(x^{\prime}, x_m)\in\R^{m-1}\times\R: x_m >0 \}$. As a consequence of \cref{(c)} and \cref{(d)}, we obtain that $(u_j)$ is uniformly bounded in $\mathrm{W}^{1,2}(\obball{1}^+)$. Since $\mathrm{W}^{1,2}(\obball{1}^+)$ is reflexive, we can assume that
\begin{equation}\label{convergence00} 
    u_j \rightharpoonup u_0 \text{ in } \mathrm{W}^{1,2}(\obball{1}^+) \text{ and } u_j\to u_0 \text{ in } \mathrm{L}^2(\obball{1}^+). 
\end{equation}

By classical trace theory, c.f. \cite[Section 5.5]{evans2010partial}, we have the following convergence 
\begin{equation}\label{trace00}
    u_j \to u_0 \text{ in } \mathrm{L}^2(\R^{m-1}\cap\obball{1}^+) \Rightarrow u_0|_{\R^{m-1}\cap\obball{1}} \equiv 0.
\end{equation}

Moreover, we also have that there exists a matrix $(Du)_0$ such that $(Du)_{r_j} \to (Du)_0$ thanks to the fact that $\{(Du)_{r_j}\}_{j\in\N}$ is equibounded.

\textbf{Step 2:} We start defining, for all $A\in\R^n\otimes\R^m$, the following sequence of operators
\begin{equation}\label{E:def_Ij}
\cI_j(A) := \frac{1}{\delta_j^2}\left[ \cI_{\cF}(\delta_j A + (Du)_{r_j}) - \cI_{\cF}((Du)_{r_j}) -\delta_j\langle D\cI_{\cF}((Du)_{r_j}),A\rangle \right],    
\end{equation}
where $\cI_{\cF}$ is defined above in \cref{AandI}. One can check that $\cI_j(A) \to D^2\cI_{\cF}((Du)_0)[A,A]$ in the $C^2$-topology. We now claim that $u_0$ is a weak solution of an elliptic system of PDEs, namely,
\begin{equation}\label{E:PDE}
\int_{\obball{1}^+} D^2\cI_{\cF}(L_u)[Du_0, D\zeta]\mathrm{d}\cH^m = 0 \text{ for all } \zeta \in C_c^{\infty}(\obball{1}^+, \R^n), \quad \zeta|_{\bp(\Gamma)} \equiv 0.
\end{equation}

For the fluency of the text, we let the proof of this claim to the end. Since $u_0$ is a weak solution of the linear elliptic system in \eqref{E:PDE}, we use boundary estimates for solutions of linear elliptic systems, see \cite[Theorem 4.2]{beck2007partial}, to get 
\begin{equation}\label{E:excess-decay:bound2der}
\begin{aligned}
    \sup_{\mathrm{B}^+_{1/2}}\| D^2 u_0 \|^2 \leq C_0\|Du_0\|^2_{\mathrm{L}^2(\mathrm{B}^+_{1})} \overset{\textup{\cref{(c)}}}{\leq} C_0.
\end{aligned}
\end{equation}

\textbf{Step 3:} By \cref{(e)}, we have that 
\begin{equation}\label{E:excess-decay:contradictionITEMe}
    C_e\varepsilon^{\alpha} \leq \frac{1}{\delta_j^2} E_{u_j}(\varepsilon, \delta_j^{-1}(L_j - \tilde{L}_j)).
\end{equation}
Our goal now is to strictly bound the right-hand side in the inequality above by $C\varepsilon^{2\alpha}$ to derive a contradiction choosing the constants appropriately. To this end, we apply the Caccioppoli inequality as follows: 
\begin{equation}\label{E:excess-decay:Caccioppoli}
\begin{aligned}
    \frac{1}{\delta_j^2} E_{u_j}(\varepsilon, \delta_j^{-1}(L_j - \tilde{L}_j)) \overset{\textup{Cor. \ref{cor:caccioppoli_ineq_graph}}}&{\leq} \frac{1}{\delta_j^2}\biggl( \frac{1}{\varepsilon^2 r_j^2}\mint_{\obball{C_u\varepsilon r_j}\cap\Omega}| u(z) - \tilde{L}_jz|^2\mathrm{d}z \biggr.\\
    &\quad \biggl. + C_0(\varepsilon r_j)^{2-\frac{2m}{p}}\|\cH_{\cF}[u]\|^2_{\mathrm{L}^p} + \kappa^2(\varepsilon r_j)^{2\alpha}\biggr).
\end{aligned}
\end{equation}
As a consequence of \eqref{E:excess_decay:contradiction_assump1}, it holds
\begin{equation}\label{E:excess-decay:MC+BDR}
    \limsup_{j\to +\infty}\left(C_0\frac{(\varepsilon r_j)^{2-\frac{2m}{p}}\|\cH_{\cF}[u]\|^2_{\mathrm{L}^p}}{\delta_j^{2}} +\kappa^2 \frac{(\varepsilon r_j)^{2\alpha}}{\delta_j^{2}}\right) \leq C_b\kappa^2\varepsilon^{2\alpha}.
\end{equation}
So, we are left to bound the limit, as $j$ goes to $+\infty$, of the integral in the right-hand side of \eqref{E:excess-decay:Caccioppoli}. To this aim, we apply a change of variables and use the definition of $u_j$ to obtain
\begin{equation}\label{E:excess-decay:termo_difficile}
    \frac{1}{\delta_j^2\varepsilon^2 r_j^2}\mint_{\obball{C_u\varepsilon r_j}\cap\Omega}| u(z) - \tilde{L}_jz|^2\mathrm{d}z = \frac{1}{\varepsilon^2} \mint_{\obball{C_u\varepsilon}\cap\Omega_j}\biggl| u_j(z) - \frac{\tilde{L}_jz - L_jz}{\delta_j}\biggr|^2\mathrm{d}z
\end{equation}
We now choose the specific $\tilde{L}_j$'s as below:
\begin{equation}\label{Ltilde}
    \tilde{L}_jz := - L_jz + \delta_j P_jz, \ \text{  where  } \ P_jz:=C_o \varepsilon^{-2} z_m \mint_{\obball{C_u\varepsilon}^+}y_mDu_j(y)\mathrm{d}y,
\end{equation}
where $C_o := C_u^2\left(\mint_{\obball{1}^+}z_m^2\mathrm{d}z\right)^{-1}$.
$\tilde{L}_j$'s can be readily proven to satisfy $T_0\Gamma\subset \mathrm{im}(h(\tilde{L}_j))$ and $\| L_j - \tilde{L}_j\| \leq C_e\frac 1j$.  We recall that $P_j$ is the unique minimizer of $P\mapsto \mint_{\obball{C_u\varepsilon}^+}|Du_j(z) - P|^2\mathrm{d}z$ among linear functions depending only on $z_m$, see \cite[Section 2]{beck2009boundary}. By \eqref{convergence00} and the choice of $\tilde{L}_j$, passing equation \eqref{E:excess-decay:termo_difficile} to the limit in $j$, we deduce
\begin{equation}\label{E:excess-decay:termo_difficile2}
\begin{aligned}
        \frac{1}{\delta_j^2\varepsilon^2 r_j^2}\mint_{\obball{C_u\varepsilon r_j}\cap\Omega}| u(z) - \tilde{L}_jz|^2\mathrm{d}z &= \frac{1}{\varepsilon^2} \mint_{\obball{C_u\varepsilon}\cap\Omega_j}\left| u_j(z) - P_jz\right|^2\mathrm{d}z \\
        \underset{j\to+\infty}&{\longrightarrow} \frac{1}{\varepsilon^2} \mint_{\obball{C_u\varepsilon}\cap\Omega_j}\left| u_0(z) - P_0z\right|^2\mathrm{d}z,
\end{aligned}
\end{equation}
where we denote $P_0z := C_o z_m \mint_{\obball{C_u\varepsilon}^+}y_mDu_0(y)\mathrm{d}y$. Combining \eqref{E:excess-decay:contradictionITEMe}, \eqref{E:excess-decay:Caccioppoli}, \eqref{E:excess-decay:MC+BDR}, and \eqref{E:excess-decay:termo_difficile2}, we get
\begin{equation*}
    C_e\varepsilon^{2\alpha} \leq C_b\kappa^2\varepsilon^{2\alpha} + \frac{1}{\varepsilon^2} \mint_{\obball{C_u\varepsilon}^+}\left| u_0(z) - P_0z\right|^2\mathrm{d}z \overset{\textup{Poinc. Ineq.}}{\leq} C_b\kappa^2\varepsilon^{2\alpha}+ \mint_{\obball{C_u\varepsilon}^+}\left| Du_0(z) - P_0\right|^2\mathrm{d}z,
\end{equation*}
where the Poincar\'e inequality can be used thanks to the fact that $u_0 - P_0$ has zero trace on $\R^{m-1}\times \{0\}$ by \eqref{trace00}. Recalling that $P_0$ is optimal, as mentioned before, we have that 
\begin{equation*}
\begin{aligned}
    \mint_{\obball{C_u\varepsilon}^+}\left| Du_0(z) - P_0\right|^2\mathrm{d}z &\leq \mint_{\obball{C_u\varepsilon}^+}\left| Du_0(z) - (Du_0)_{C_u\varepsilon}\right|^2\mathrm{d}z \\
    &\leq C_u^2\varepsilon^2\mint_{\obball{C_u\varepsilon}^+}\left| D^2u_0(z)\right|^2\mathrm{d}z \overset{\eqref{E:excess-decay:bound2der}}{\leq} C_0C_u^2\varepsilon^2.
\end{aligned}
\end{equation*}
Finally, the last two displayed inequalities provide $C_e \leq C_b\kappa^2 + C_0C_u^2$ which in turn implies the desired contradiction once we choose $C_e$ large enough.

\textbf{Proof of the claim \eqref{E:PDE}:} Without loss of generality we can assume that $\bp(\Gamma) = \obball{r_0}\cap\R^{m-1}$. Indeed, it is a standard procedure of straightening/flattening out the boundary, see \cite[Subsection 3.2.3]{evans2010partial}. If the boundary is not flat, i.e., $\bp(\Gamma) \neq \obball{r_0}\cap\R^{m-1}$, we take a smooth function $\Phi$ such that $\Phi(0)=0, D\Phi(0)= 0$, and $\Phi(\bp(\Gamma)) = \obball{r_0}\cap\R^{m-1}$. So, $u_0\circ\Phi$ satisfies \eqref{E:PDE}, which assures that $u_0$ satisfies a similar elliptic PDE. For more details on this standard argument, we refer the reader to \cite[Subsection 3.2.3]{evans2010partial}.

Fix a flattened boundary $\bp(\Gamma)= \obball{r_0}\cap\R^{m-1}$, denote $\mathrm{B}^+_s:= \obball{s}\cap \{(x^{\prime}, x_m)\in\R^{m-1}\times\R: x_m >0 \}$ for every $s>0$, and $q\in\R$ the conjugate exponent of $p$, i.e., such that $p^{-1}+q^{-1}=1$. Let $\zeta\in C^{\infty}_c(\mathrm{B}^+_1,\R^n)$ a test vector field with $\zeta(z^{\prime}, 0) = 0$ for every $z^{\prime}\in\R^{m-1}$. Then we define the sequence $\zeta_j(z) := \zeta(\frac{z}{r_j})$ which for each $j$ also satisfies $\zeta_j(z^{\prime}, 0) = 0$ for every $z^{\prime}\in\R^{m-1}$. 

Our aim now is to apply \cref{L:technical}. To this aim, we estimate the left-hand side of \eqref{useful0} in \cref{L:technical} as follows
\begin{equation}\label{E:CofV}
\begin{aligned}
        \int_{\mathrm{B}^+_1} \langle D\cI_{\cF}(Du(z)),        D\zeta_j(z)\rangle\mathrm{d}z &= r_j^{-1}\int_{\mathrm{B}^+_1} \langle D\cI_{\cF}(Du(z)), D\zeta\left(\frac{z}{r_j}\right)\rangle\mathrm{d}z \\
        &=r_j^{m-1}\int_{\mathrm{B}^+_1} \langle D\cI_{\cF}(Du(r_j z)), D\zeta\left(z\right)\rangle\mathrm{d}z.
\end{aligned}
\end{equation}

Notice that the domain of integration does not change under the change of variables since $\zeta$ has compact support. Thus, by \eqref{E:CofV}, we obtain that 
\begin{equation}\label{E:LHS}
\begin{aligned}
    \int_{\mathrm{B}^+_1} \langle D\cI_{\cF}(Du(z)), &D\zeta_j(z)\rangle\mathrm{d}z = r_j^{m-1}\int_{\mathrm{B}^+_1} \langle D\cI_{\cF}(Du(r_j z)) -  D\cI_{\cF}((Du)_{r_j}), D\zeta\left(z\right)\rangle\mathrm{d}z \\
    \overset{\eqref{(a)}}&{=} r_j^{m-1}\int_{\mathrm{B}^+_1} \langle D\cI_{\cF}((Du)_{r_j} + \delta_j Du_j(z)) -  D\cI_{\cF}((Du)_{r_j}), D\zeta\left(z\right)\rangle\mathrm{d}z\\
    \overset{\eqref{E:def_Ij}}&{=} \delta_j r_j^{m-1}\int_{\mathrm{B}^+_1} \langle D\cI_j(Du_j(z)), D\zeta\left(z\right)\rangle\mathrm{d}z,
\end{aligned}
\end{equation}
where we used the compactness of the support of $\zeta$ and the divergence theorem for the first equality.
We now focus on the right-hand side of \eqref{useful0}. Recalling the definition of $\cA$ and that $u$ is Lipschitz, we have that $\|\zeta_j \cA^{1/q}(Du) \|_{\mathrm{L}^q(\obball{r}\cap\Omega)} \leq C_0 \|\zeta_j\|_{\mathrm{L}^q(\obball{r}\cap\Omega)}$, where we change variables to get that
\begin{equation}\label{E:RHS}
\|\zeta_j \cA^{1/q}(Du) \|_{\mathrm{L}^q(\obball{r}\cap\Omega)} \leq C_0 r_j{^\frac{m}{q}} \|\zeta\|_{\mathrm{L}^q(\obball{r}\cap\Omega)}.
\end{equation}

We now use \cref{L:technical}, \eqref{E:LHS}, and \eqref{E:RHS}, to derive that
\begin{equation}\label{E:LHS+RHS}
\begin{aligned}
\delta_j r_j^{m-1}\int_{\mathrm{B}^+_1} \langle D\cI_j(Du_j(z)), D\zeta\left(z\right)\rangle\mathrm{d}z &= \delta_j r_j^{m-1}\int_{\mathrm{B}^+_1} \langle D\cI_{\cF}(Du(z)), D\zeta_j(z)\rangle\mathrm{d}z \\
&\leq C^{\prime}\|\zeta_j \cA^{1/q}(Du) \|_{\mathrm{L}^q(\obball{r}\cap\Omega)} \\
&\leq C_0^{\prime} r_j{^\frac{m}{q}} \|\zeta\|_{\mathrm{L}^q(\obball{r}\cap\Omega)}.
\end{aligned}
\end{equation}

Recalling \eqref{E:excess_decay:contradiction_assump1} and the choice of $q$, we easily obtain that
\begin{equation*} C_0^{\prime} \frac{r_j^{\frac{m}{q}-m+1}}{\delta_j} = C_0^{\prime} \frac{r_j^{1 - \frac{m}{p}}}{\delta_j} \leq C_0^{\prime} \frac{\delta_j}{\|\cH_{\cF}[u]\|_{\mathrm{L}^p(\Omega\times\R^n)}},\end{equation*}
which in turn, together with \eqref{E:LHS+RHS}, implies that 
\begin{equation*} \lim_{j\to +\infty}\int_{\mathrm{B}^+_1} \langle D\cI_j(Du_j(z)), D\zeta\left(z\right)\rangle\mathrm{d}z = 0.\end{equation*}

By the very same argument of \cite[Proposition 4.5]{de2020regularity}, we conclude from the previous equation that
\begin{equation*} \int_{\mathrm{B}^+_1}D^2\cI_{\cF}(L_u)[Du_0, D\zeta] = 0, \end{equation*}
for all $\zeta\in C^{\infty}_c(\mathrm{B}^+_1, \R^n)$ with $\zeta(x^{\prime}, 0) = 0, \forall x^{\prime}\in\obball{1}\cap\R^{m-1}$, as claimed in \eqref{E:PDE}.
\end{proof}

\section{Boundary regularity}

By \cref{P:excess-decay}, there exists $C_e=C_e(m,n,\|\cF\|_{C^2}, K_{\cF}, \|u\|_{\mathrm{Lip}}, \Gamma)>0$ with the following property: setting $\gamma:=\min\{\alpha,1-m/p\}$ and
\begin{equation}\label{E:reg:choice-eps}
    \varepsilon < \min\left \{ (12 C_e)^{-\frac{1}{2\alpha}}, C_e^{-\frac{1}{2\alpha - \gamma}}, 2^{-\frac{2}{\gamma}}, 8^{-\frac{1}{\gamma}},\frac{1}{4C_u} \right\},
\end{equation}
then there exists $\delta_* = \delta_*(\varepsilon)>0$ such that for every $\delta \in (0,\delta_*)$ and $r>0$,
\begin{equation}\label{E:reg:excess-decay}
    \begin{cases} \text{if }L\in\R^m\otimes\R^n, T_0\Gamma\subset\mathrm{im}(h(L)), \\ r^{\gamma}\|\cH[u]\|_{p} \leq E(r, L) \leq \delta,\end{cases}\Longrightarrow \begin{cases} \exists \tilde{L}\in \R^m\otimes\R^n\text{ s.t. }T_0\Gamma\subset \mathrm{im}(h(\tilde{L})), \\ \| L- \tilde{L}\| \leq C_e \delta, \\ E(\varepsilon r,\tilde{L}) \leq C_e\varepsilon^{2\alpha} E(r, L). \end{cases}
\end{equation}

We now define the auxiliary excess, which encompasses the mean curvature rather than only the excess $E$, as follows
\begin{equation*}
\begin{aligned}
    e(x, s, L) := E(x, s, L) + \frac{8s^{\min\{\alpha, 1-\frac{m}{p}\}}}{\varepsilon^{m}}\|\cH[u]\|_{p}, \ e(s, L):= e(0,s, L), \ \forall s>0.    
\end{aligned}
\end{equation*}

We prove a decay for $e$ in the next corollary, which is a consequence of the excess decay, \cref{P:excess-decay}. We lastly choose
\begin{equation}\label{E:reg:choice:r1-eps}
    r_1:=\min\left\{\frac{\delta^{1/2\alpha}_*}{6C_d}, \left(\frac{\varepsilon^m\delta_*}{2\|\cH[u]\|_p}\right)^{\frac{1}{\gamma}}\right\} \text{ and } \sigma < \frac{\delta_*}{6C_d}.
\end{equation}

\begin{corollary}\label{C:decay-aux-excess}
Assume \cref{assump:H_u-L_p}, \eqref{E:reg:choice-eps}, and \eqref{E:reg:choice:r1-eps}. Then, for each $j\in\N$, there exists $L_j$ with $T_0\Gamma\subset\mathrm{im}(h(L_j))$ such that $$e(\varepsilon^j r, L_j) \leq 2^{-2j}e(r, L_u)\quad \mbox{ for every $r\in(0,r_1)$},$$
where $\| L_{j+1} - L_{j}\| \leq 2^{-2j} C_e \delta_*$. Hence, there exist $L_\infty := \lim_j L_j$ with $T_0\Gamma\subset\mathrm{im}(h(L_\infty))$, $\eta\in(0,1)$, and $c_e = c_e(r_1,\varepsilon)>0$ such that 
 $$e(r,L_\infty) \leq c_e r^{2\eta} e(r_1, L_u) \quad \mbox{for all $r\in (0,r_1)$}.$$
\end{corollary}
\begin{proof}
By our choice of $r_1$ and $\sigma$ in \eqref{E:reg:choice:r1-eps}, performing the same computations of \cref{L:density}, we have that for every $r<r_1$ it holds
\begin{equation}\label{E:reg-smallness-excess-H}
    e(r, L_u) = \frac{8r^{\gamma}}{\varepsilon^m}\|\cH[u]\|_{p} + E(r, L_u) \leq \delta_*.
\end{equation}

We wish to prove the existence of $L_1\in\R^m\otimes\R^n$ with $T_0\Gamma\subset \mathrm{im}(h(L_1))$ and $\| L_u - L_1\| \leq C_e\delta_*$ such that
\begin{equation}\label{E:reg:decay-aux}
    e(\varepsilon r, L_1) \leq 2^{-2} e(r, L_u).
\end{equation}

To this end we consider two cases.

Case 1: if $\varepsilon^{-m}r^{\gamma}\|\cH[u]\|_{p}\leq E(r, L_u)$, we can apply \eqref{E:reg:excess-decay} to deduce the existence of such $L_1$ satisfying that
\begin{align*}
    e(\varepsilon r,L_1) \overset{\eqref{E:reg:excess-decay}}&{\leq} C_e\varepsilon^{2\alpha}E(r, L_u) + 8\frac{(\varepsilon r)^{\gamma}}{\varepsilon^{m}}\|\cH[u]\|_{p} \leq (C_e\varepsilon^{2\alpha-\gamma})\varepsilon^{\gamma}E(r, L_u) + 8\frac{(\varepsilon r)^{\gamma}}{\varepsilon^{m}}\|\cH[u]\|_{p} \\
    \overset{(\ast)}&{\leq} \varepsilon^{\gamma}e(r, L_u) \leq 2^{-2} e(r, L_u) ,
\end{align*}
which is precisely \eqref{E:reg:decay-aux}. Here $(\ast)$ follows from the fact that $\varepsilon < \min\{C_e^{-\frac{1}{2\alpha - \gamma}},2^{-2/\gamma}\}$, as assumed in \eqref{E:reg:choice-eps}. 

Case 2: if $r^{\gamma}\|\cH[u]\|_{p} > E(r, L_u)$, we choose $L_1 = L_u$ and we have that:
\begin{align*}
    e(\varepsilon r, L_u) &\leq \varepsilon^{-m}E( r,L_u) + 8\varepsilon^{-m+\gamma}r^{\gamma}\|\cH[u]\|_{p} < \left( \varepsilon^{-m}r^{\gamma} + 8\varepsilon^{-m+\gamma}r^{\gamma}\right)\|\cH[u]\|_{p}\\
    &= \left( \frac18 + \varepsilon^{\gamma} \right) \frac{8r^{\gamma}}{\varepsilon^{m}}\|\cH[u]\|_{p}
    \overset{(\ast\ast)}{\leq} \frac14 \frac{8r^{\gamma}}{\varepsilon^{m}}\|\cH[u]\|_{p} \leq 2^{-2}e(r, L_u),
\end{align*}
which is exactly \eqref{E:reg:decay-aux}. In $(\ast\ast)$, we used that $\varepsilon < 8^{-1/\gamma}$, as we have assumed in \eqref{E:reg:choice-eps}. 


Since $e(\varepsilon r, L_1) \leq 2^{-2}e(r,L_u) \leq 2^{-2}\delta_*$, we can reiterate the previous argument with $\delta=2^{-2}\delta_*$ to obtain the existence of $L_2$ with $T_0\Gamma\subset \mathrm{im}(h(L_2))$ and $\| L_2 - L_1\| \leq C_e2^{-2}\delta_*$ such that $e(\varepsilon^2 r, L_2) \leq 2^{-4}e(r,L_u) $. Arguing inductively, for each $j\in\N$, we deduce the existence of $L_j$ with $T_0\Gamma\subset \mathrm{im}(h(L_j))$ and $\| L_{j+1} - L_{j}\| \leq 2^{-2j} C_e \delta_*$ such that $e(\varepsilon^j r, L_j) \leq 2^{-2j}e(r,L_u)$. Clearly the estimate $\| L_{j+1} - L_{j}\| \leq 2^{-2j} C_e \delta_*$ guarantees the existence of $L_\infty := \lim_j L_j$ with the desired properties.
\end{proof}

We finally have all the tools to state and prove our main theorem. We will rewrite all the assumptions made up to now as part of the hypothesis of the theorem for the reader's convenience.

\begin{theorem}[Boundary regularity theorem]\label{T:regularity}
Let $m,n\geq 2$, $\cF$ be an integrand of class $C^2$ on the $m$-Grasmannian bundle $\Gr(\ball{x}{r_0})$ satisfying USAC, $\Gamma$ be an $(m-1)$-submanifold of class $C^{1,\alpha}$ in $\ball{x}{r_0}$ with reach $\kappa\leq (2r_0^{\alpha})^{-1}$ and such that $x\in\Gamma$. Let $\Omega$ be an open subset of $\ball{x}{r_0}\cap (\R^m\times\{0\})$ and $\bV=\bV[u]\in \V_m(\ball{x}{r_0})$ be an $m$-varifold induced by the graph of $u\in Lip(\Omega,\R^n)$, and $\partial \mathrm{graph}(u) = \Gamma$. Assume that the anisotropic first variation $\delta_{\cF}\bV$ is a Radon measure on $\ball{x}{r_0}\setminus\Gamma$ and the anisotropic mean curvature $\cH_{\cF}\in\mathrm{L}^p(\ball{x}{r_0}), p>m$. Then there exist three constants $\sigma>0$, $\beta \in (0,1)$ and $\zeta \in(0,1)$ depending only on $m,n,p,\|\cF\|_{C^2}$, $K_{\cF}, \|u\|_{\mathrm{Lip}}, \Gamma$, such that, if
$$\|\bV\|(\ball{x}{r})\leq \left(\frac12 + \sigma\right)\omega_m r^m, \quad \mbox {for any $r\in (0, r_0)$},$$  then $u\in C^{1,\zeta}(\bball{x}{\beta r_0})$.
\end{theorem}
\begin{proof}
Without loss of generality, we can assume $x=0$. Denote $\gamma := \min\{\alpha, 1-m/p\}$. The hypothesis of this theorem match exactly \cref{assump:H_u-L_p}. We can choose $\varepsilon, r_1, \sigma >0$ satisfying \eqref{E:reg:choice-eps} and \eqref{E:reg:choice:r1-eps}. We apply Lemma \ref{L:density} to obtain the existence of $L_u$ and we recall that the excess $E(\cdot, r, L_u)$ is continuous with respect to the variable in $\R^{m+n}$. Hence, as in the proof of \eqref{E:reg-smallness-excess-H}, there exists $\beta \in (0,1)$ such that
\begin{equation*}
    e(y,r,L_u) = \frac{8r^{\gamma}}{\varepsilon^m}\|\cH[u]\|_p + E(y,r, L_u)\leq \delta_*(\varepsilon), \quad \forall y\in \oball{\beta r_0}, \quad \forall r\in (0,r_1).
\end{equation*}

We apply \cref{C:decay-aux-excess} to obtain the existence of $L_\infty$, $\eta\in(0,1)$, and $c_e>0$, such that
\begin{equation*}
    e(y,r, L_\infty)\leq c_e r^{2\eta}e(y,r_1,L_u), \quad \forall y\in \oball{\beta r_0}, \quad \forall r\in (0,r_1).
\end{equation*}

This shows that $Du$ restricted to $\obball{\beta r_0}$ belongs to a Campanato space, for a possibly smaller (not relabeled) $\beta$ depending on $\|u\|_{\mathrm{Lip}}$, hence it is a $\zeta$-H\"older continuous function (\cite[Chapter 1, Lemma 1]{simon1996theorems}) for some $\zeta\in(0,1)$, which concludes the proof of the theorem. 
\end{proof}

\renewcommand{\abstractname}{Acknowledgments}
\begin{abstract}
Antonio De Rosa has been partially supported by the NSF DMS CAREER Award No.~2143124. 
Reinaldo Resende was supported by CAPES-Brazil with a Ph.D. scholarship 88882.377954/2019-01. Reinaldo Resende carried out part of this work in Princeton University (this visit was financed by FAPESP-Brazil grant 2021/05256-0) and The Fields Institute for Research in Mathematical Sciences during the Thematic Program on Nonsmooth Riemannian and Lorentzian Geometry. Both authors warmly thank Stefano Nardulli for many interesting discussions.
\end{abstract}

\addcontentsline{toc}{section}{References}
\bibliographystyle{acm}
\bibliography{9biblio}
\end{document}